\documentclass[a4paper,12pt,leqno]{amsart}
\usepackage{amsmath,amssymb,mathrsfs}
\usepackage{xcolor,hyperref}
\hypersetup{colorlinks=true,citecolor=black,linkcolor=black}
\newcommand{\h}{\hbox}
\newcommand{\q}{\quad}
\newcommand{\nin}{\noindent}

\newcommand{\sk}{\par\smallskip}

\newcommand{\skn}{\par\smallskip\noindent}
\newcommand{\ges}{\geqslant}
\newcommand{\les}{\leqslant}
\newcommand{\one}{\hskip1pt}

\newcommand{\msum}{\hbox{$\sum$}}

\newcommand{\mprod}{\hbox{$\prod$}}

\newcommand{\J}{{\mathcal J}}

\newcommand{\OO}{{\mathcal O}}

\newcommand{\Q}{{\mathbb Q}}
\newcommand{\C}{{\mathbb C}}
\newcommand{\N}{{\mathbb N}}
\newcommand{\R}{{\mathbb R}}

\newcommand{\Z}{{\mathbb Z}}

\newcommand{\xb}{{\mathbf x}}
\newcommand{\yb}{{\mathbf y}}
\newcommand{\Ut}{\widetilde{U}}
\newcommand{\Xt}{\widetilde{X}}
\newcommand{\chit}{\widetilde{\chi}}
\newcommand{\bo}{{\bf 1}}
\newcommand{\Gap}{\Gamma_{\!+}}

\newcommand{\al}{\alpha}
\newcommand{\be}{\beta}
\newcommand{\ga}{\gamma}

\newcommand{\si}{\sigma}

\newcommand{\ep}{\varepsilon}

\newcommand{\vp}{\varphi}
\newcommand{\dd}{\partial}
\newcommand{\ddd}{{\rm d}}

\newcommand{\tos}{\,{\to}\,}
\newcommand{\eq}{\,{=}\,}
\newcommand{\defs}{\,{:=}\,}
\newcommand{\nes}{\,{\ne}\,}
\newcommand{\ins}{\,{\in}\,}
\newcommand{\sst}{\,{\subset}\,}
\newcommand{\stm}{\,{\setminus}\,}
\newcommand{\gess}{\,{\ges}\,}
\newcommand{\less}{\,{\les}\,}
\newcommand{\sgt}{\,{>}\,}
\newcommand{\slt}{\,{<}\,}
\newcommand{\col}{\,{:}\,}
\newcommand{\pl}{\one {+}\one}
\newcommand{\mi}{\one {-}\one}
\newcommand{\bl}{\bigl}
\newcommand{\br}{\bigr}

\newcommand{\ssb}{\raise.15ex\h{${\scriptscriptstyle\bullet}$}}
\newcommand{\ssc}{\,\raise.15ex\h{${\scriptstyle\circ}$}\,}

\newcommand{\into}{\hookrightarrow}
\newcommand{\simto}{\,\,\rlap{\hskip1.5mm\raise1.4mm\hbox{$\sim$}}\hbox{$\longrightarrow$}\,\,}

\makeatletter
\renewcommand\section{\@startsection{section}{1}{0pt}{-3ex plus -1ex minus -.2ex}{2.3ex plus.2ex}{\centering\normalfont\bfseries}}
\makeatother
\theoremstyle{plain}
\newtheorem{thm}{Theorem}[section]
\newtheorem{cor}{Corollary}[section]

\newtheorem{ithm}{Theorem}

\theoremstyle{definition}
\newtheorem{rem}{Remark}[section]

\begin{document}
\title[Subtlety of oscillation indices]
{Subtlety of oscillation indices of oscillatory integrals of real analytic functions}
\author[I.-K. Kim]{In-Kyun Kim}
\address{I.-K. Kim : June E Huh Center for Mathematical Challenges, Korea Institute for Advanced Study, 85 Hoegiro Dongdaemun-gu, Seoul 02455, Korea}
\email{soulcraw@kias.re.kr}
\author[M. Saito]{Morihiko Saito}
\address{M. Saito : RIMS Kyoto University, Kyoto 606-8502 Japan}
\email{msaito@kurims.kyoto-u.ac.jp}
\thanks{This work was partially supported by National Research Foundation of Korea NRF-2023R1A2C1003390 and NRF-2022M3C1C8094326.}
\begin{abstract} For a locally defined real analytic function $f$, we study the relation between the oscillation index of oscillatory integrals and the real log canonical threshold. The former is always negative, and its absolute value is greater than or equal to the latter. They coincide very often, but there are certain exceptional cases, and it is not very clear when the equality holds. In this note we give some sufficient conditions for the coincidence to hold or to fail. In the Newton-nondegenerate convenient homogeneous case, we show that the strict inequality holds if the number of variables $n$ is even and smaller than the degree $d$ of $f$ (or $f^{-1}(0)=\{0\}$), and the equality holds if $n$ is odd and $f^{-1}(0)=\{0\}$ (in particular, $d$ is even). The first assertion does not seem to be compatible with some standard formula in the literature, and there must be some error somewhere, although it does not seem easy to detect it inside this paper.
\end{abstract}
\maketitle

\section*{Introduction} \label{intr}
\nin
Let $f$ be a real analytic function defined on a small open neighborhood $U$ of $0$ in $X\defs\R^n$ with $f(0)\eq0$ and $n\gess2$. For $C^{\infty}$ functions $\vp$ on $U$ with compact supports, we can consider the {\it oscillatory integrals}
\begin{equation} \label{1}
I(\tau,\vp)\defs\int_Xe^{\sqrt{-1}\,\tau f(x)}\vp(x)\ddd\xb\q\q\h{for}\,\,\,\tau\ins\R_{>0},
\end{equation}
where $\ddd\xb\defs\ddd x_1{\wedge}\cdots{\wedge}\ddd x_n$ with $\xb\eq(x_1,\dots,x_n)$ the canonical coordinate system of $\R^n$. As a consequence of Hironaka's resolution of singularities (see \cite{H}), we have the {\it asymptotic expansions\one} for $\tau\to+\infty$\,:
\begin{equation} \label{2}
I(\tau,\vp)\sim\sum_{\al<0}\sum_{k=0}^{n-1}C_{\al,k}(\vp)\tau^{\al}(\log\tau)^k\q\q\h{with}\,\,\,\,C_{\al,k}(\vp)\ins\C,
\end{equation}
 see \cite{M}, \cite{V}, \cite{AGV}, \cite{CKN}. The {\it oscillation index\one} $\be(f)$ of $f$ at $0\ins X$ is by definition the maximal number $\be$ such that for any neighborhood $V$ of 0 in $U$, there is a $C^{\infty}$ function $\vp$ compactly supported in $V$ with $C_{\be,k}(\vp)\nes0$ for some $k$.
\sk
On the other hand, we have the {\it real log canonical threshold\one} ${\rm rlct}(f)$ at 0 for a real analytic function $f$ defined on $U$, which can be determined also by using a resolution of singularities, see \cite{rlct}. Indeed, let $\pi\col\Ut\to U$ be a resolution of singularities for $f$, which is obtained by repeating smooth center blowups, and such that $\pi^*\!f$ and its Jacobian determinant are monomials up to multiplication by unit for some local coordinates at each point of $\pi^{-1}f^{-1}(0)$, see \cite{H}. Let $D_j$ ($j\ins J$) be the irreducible components of $D\defs\pi^{-1}f^{-1}(0)$, which we may assume to be smooth. Let $m_j,k_j$ be the multiplicities of $\pi^*\!f$ and $\pi^*\ddd\xb$ at general points of $D_j$. Set
\begin{equation} \label{3}
\ga(f)\defs\min_{j\ins J}\bl\{\tfrac{k_j{+}1}{m_j}\br\}.
\end{equation}
Note that $\ga(f)\less1$ if $f^{-1}(0)$ is $(n{-}1)$-dimensional, and all the $D_j$ for $j\ins J$ are exceptional divisors if $\dim f^{-1}(0)\slt n{-}1$. We have the following
\begin{ithm}[\cite{rlct}] \label{T1}
There is an equality
\begin{equation} \label{4}
{\rm rlct}(f)=\ga(f).
\end{equation}
\end{ithm}
Note that the real log canonical threshold can be greater than 1 and also than the log canonical threshold ${\rm lct}(f_{\C})$ in the case $\dim f^{-1}(0)\slt n{-}1$, where $f_{\C}$ is the complexification of $f$ obtained by the inclusion $\R\{\xb\}\into\C\{\xb\}$. 
\sk
As for the oscillation index, we have the following.
\begin{ithm}[\cite{V}] \label{T2}
There is an inequality
\begin{equation} \label{5}
\be(f)\less{-}\ga(f).
\end{equation}
\end{ithm}
Note that the normal crossing case is a very special case of a Newton-nondegenerate case, see also \cite{CKN}. It seems to be believed by some people that the equality holds in \eqref{5} at least in the Newton-$\R$-nondegenerate case with $\ga(f)\slt1$, see for instance \cite{AGV}, \cite{V}.
\sk
In this note we show the following.
\begin{ithm} \label{T3}
Let $f$ be a real homogeneous polynomial which is Newton-$\R$-nondegenerate and convenient.
\skn
{\rm(i)} If $n$ is even and either $d\defs\deg f\sgt n$ or $f^{-1}(0)\eq\{0\}$ $($in particular, $d$ is even$)$, then
\begin{equation} \label{6}
\be(f)<-\ga(f)=-\tfrac{n}{d}.
\end{equation}
\skn
{\rm(ii)} If $n$ is odd and $f^{-1}(0)\eq\{0\}$ $($in particular, $d$ is even$)$, then
\begin{equation} \label{7}
\be(f)=-\ga(f)=-\tfrac{n}{d}.
\end{equation}
\end{ithm}

Note that if $f$ is a real homogeneous polynomial and $d$ is {\it odd,} then $\dim f^{-1}(0)\eq n{-}1$, see Remark\,\,\ref{R3.1} below.
Theorem\,\,\ref{T3} does not seem compatible with some well-known formula in the literature, see for instance \cite{V}, \cite{CKN}. So there must be some error somewhere, although it does not seem easy to detect it inside this paper.
\sk
In Section~1, we review some basics of (real) jumping numbers and (real) log canonical thresholds. In Section~2 we recall some theorem from the theory of oscillatory integrals. In Section~3 we prove Theorem~\ref{T3} using the assertions in the previous section.

\tableofcontents
\numberwithin{equation}{section}

\section{Log canonical thresholds} \label{S1}
\nin
In this section we review some basics of (real) jumping numbers and (real) log canonical thresholds.
\sk
Let $f\ins\R\{\xb\}\eq\R\{x_1,\dots,x_n\}$. This real analytic function is defined on a small open subset $U$ of $X\defs\R^n$. We have the {\it real multiplier ideals\one} $\J(U,f^{\al})\sst\OO_U$ for $\al\ins\R_{\ges0}$, which is defined by the {\it local integrability\one} of $|g|/|f|^{\al}$ for $g\in\OO_U$.
(Here coherence of $\J(U,f^{\al})$ is unclear.) These are weakly decreasing for $\al$, and
coincide with $\OO_U$ for $0\less\al\,{\ll}\,1$. It is known that there is an increasing sequence of positive rational numbers $\al_k$ ($k\ins\Z_{>0}$), called the {\it real jumping numbers\one} of $f$, such that
\begin{equation} \label{1.1}
\J(U,f^{\be})=\J(U,f^{\al_k})\ne\J(U,f^{\al_{k+1}})\q\h{if}\,\,\,\be\ins[\al_k,\al_{k+1}),\,k\ins\N,
\end{equation}
where $\al_0\eq0$. The minimal real jumping number is called the {\it real log canonical threshold,} and is denoted by ${\rm rlct}(f)$, see \cite{rlct} (where $\R\{\!\{x\}\!\}$ means $\R\{\xb\}$).

\begin{rem} \label{R1.1}
We do {\it not\one} necessarily have
\begin{equation} \label{1.2}
f\J(U,f^{\al})=\J(U,f^{\al+1})\q\h{for}\,\,\,\al\sgt0,
\end{equation}
in the exceptional type case. Here we say that $f$ is {\it exceptional type} if ${\rm codim}_U\,D\sgt1$ with $D\defs f^{-1}(0)\sst U$, and {\it ordinary type\one} otherwise.
\end{rem}

\begin{rem} \label{R1.2}
The above equality \eqref{1.2} always holds for the complexification $f_{\C}$ of $f$, which is defined on a sufficiently small open neighborhood $U_{\C}$ of 0 in $\C^n$ with $U_{\C}\cap\R^n\eq U$ using the inclusion $\R\{\xb\}\into\C\{\xb\}$. The multiplier ideals $\J(U_{\C},f_{\C}^{\,\al})$ can be defined by using the local integrability of $|g|^2/|f_{\C}|^{2\al}$ for $g\ins\OO_{U_{\C}}$, and they are closely related to the $V$-filtration of Kashiwara and Malgrange {\it indexed by\one} $\Q$, see for instance \cite{BS}.
\end{rem}

\begin{rem} \label{R1.3}
The minimal jumping number is called the {\it log canonical threshold,} which is denoted by ${\rm lct}(f_{\C})$. This coincides with the minimal spectral number in the isolated singularity case if the latter is at most 1, see for instance \cite{JKSY}. We have an inequality
\begin{equation} \label{1.3}
{\rm lct}(f_{\C})\les{\rm rlct}(f),
\end{equation}
where the equality does not necessarily hold, see \cite{rlct} (and also Remark\,\,\ref{R3.3} below).
For the relation to the maximal root of Bernstein-Sato polynomial up to sign, see \cite{K} (and \cite[Remark 1.2 (iii)]{JKSY}).
\end{rem}

\section{Oscillatory integrals} \label{S2}
\nin
In this section we recall some theorem from the theory of oscillatory integrals.
\sk
For $f,\vp$ as in the introduction, let $\Gap(f),\Gap(\vp)$ be their {\it Newton polytopes\one} defined by using their Taylor expansions with coordinates $x_1,\dots,x_n$ fixed; for instance $\Gap(f)$ is the convex hull of the union of $\nu\pl\R_{\ges0}^{\,n}\sst\R_{\ges0}^{\,n}$ for $\nu\ins\N^n$ with $a_{\nu}\nes0$, where the $a_{\nu}$ are the coefficients of the Taylor expansion $\msum_{\nu\in\N^n}\,a_{\nu}\xb^{\nu}$ of $f$ with $\xb^{\nu}\defs\mprod_{i=1}^n\,x_i^{\nu_i}$.
\sk
A real analytic function $f$ is called {\it Newton-$\R$-nondegenerate\one} if for any compact face $\si$ of $\Gap(f)$, the intersection of $\{\dd_{x_i}f_{\si}\eq0\}\sst\R^n$ for $i\ins[1,n]$ is contained in $\{x_1\cdots x_n\eq0\}$, where $f_{\si}\defs\msum_{\nu\in\si}\,a_{\nu}\xb^{\nu}$ is the partial sum of the Taylor expansion of $f$ over $\si\cap\N^n$. We define $\C$-nondegeneracy considering the zero-locus in $\C^n$ instead of $\R^n$. We say that $f$ is {\it convenient\one} if the Newton polytope $\Gap(f)$ intersects every coordinate axis of $\R^n$, see \cite{V}, \cite{CKN}.
\sk
Assume $\Gap(f),\Gap(\vp)$ are non-empty, or equivalently, their Taylor expansions are nonzero. Set
\begin{equation*}
d(f,\vp):=\min\bl\{d\ins\R_{>0}\mid d\,{\cdot}\,(\Gap(\vp){+}\bo)\sst\Gap(f)\br\},
\end{equation*}
where $\bo\defs(1,\dots,1)\ins\R^n$. Set $(\R_{\ges0}^{\,n})^{\ges r}\defs\{\nu\ins\R_{\ges0}^{\,n}\mid|\nu|\gess r\}$ with $|\nu|\defs\msum_{i=1}^n\nu_i$.
\sk
If $f$ is convenient so that
\begin{equation*}
\Gap(f)\supset(\R_{\ges0}^{\,n})^{\ges r},\q\Gap(\vp)\subset(\R_{\ges0}^{\,n})^{\ges r'},
\end{equation*}
for some $r,r'\ins\R_{>0}$, then we have
\begin{equation} \label{2.1}
d(f,\vp)\les\frac{r}{r'{+}\,n}.
\end{equation}
Note that $r\eq d$ if $f$ is a convenient homogeneous polynomial of degree $d$.
\sk
Let $\be(f,\vp)$ be the maximal number $\be$ such that $C_{\be,k}(\vp)\nes0$ for some $k$ in \eqref{1}, see \cite{CKN}. We have the following.

\begin{thm}[{\cite[Theorem 2.2]{CKN}}] \label{T2.1}
Assume $f$ is convenient, Newton-$\R$-nondegenerate, and $\Gap(\vp)\nes\emptyset$. Let $r,r'$ be as in {\rm\eqref{2.1}}. Then we have the inequalities
\begin{equation} \label{2.2}
\be(f,\vp)\les-1/d(f,\vp)\les-\frac{r'{+}\one n}{r}.
\end{equation}
\end{thm}

We say that a compactly supported $C^{\infty}$ function $\chi$ on $U$ is a {\it cutoff function\one} if the values of $\chi$ are nonnegative and coincide with 1 on a neighborhood of 0 in $U$. (The support of $\vp$ is the closure of the nonzero locus in $U$.) Theorem\,\,\ref{T2.1} implies the following.

\begin{cor} \label{C2.1}
Assume $f$ is Newton-$\R$-nondegenerate and convenient. Then the asymptotic expansion of $I(\tau,\xb^{\nu}\chi)$ is independent of the choice of a cutoff function $\chi$ for any $\nu\ins\N^n$.
\end{cor}

\begin{proof}
Let $\vp_k$ be compactly supported $C^{\infty}$ functions on $U$ such that $(\R_{\ges0}^{\,n})^{\ges k}\,{\supset}\,\Gap(\vp_k)\nes\emptyset$ for $k\,{\gg}\,0$. This corollary can be verified by applying Theorem\,\,\ref{T2.1} to the difference between $\xb^{\nu}\chi\pm\vp_k$ and $\xb^{\nu}\chi$, and also to the one between $\xb^{\nu}\chi\pl\vp_k$ and $\xb^{\nu}\chi'\mi\vp_k$ for two cutoff functions $\chi$ and $\chi'$, for any $k\,{\gg}\,0$.
\end{proof}

\begin{cor} \label{C2.2}
Under the hypothesis of Corollary\,\,{\rm\ref{C2.1}}, the asymptotic expansion of $I(\tau,\vp)$ depends only on the Taylor expansion $\msum_{\nu\in\N^n}\,c_{\nu}\xb^{\nu}$ of $\vp$.
\end{cor}

\begin{proof}
This follows from Corollary\,\,\ref{C2.1} applying Theorem\,\,\ref{T2.1} to $\vp-\msum_{|\nu|\les k}\,c_{\nu}\xb^{\nu}\chi$ for any $k\gg 0$.
(Note that the Taylor expansion of $\vp$ is not necessarily a convergent power series.)
\end{proof}

\section{Proof of the main theorem} \label{S3}
\nin
In this section we prove Theorem~\ref{T3} using the assertions in the previous section.

\begin{proof}[Proof of Theorem~{\rm\ref{T3} (i)}]
Let $\pi\col\Xt\tos X$ be the blowup at $p\ins X$. The embedded resolution of singularities is given by this. Set $E\defs\pi^{-1}(0)$. For $i\ins[1,n]$, there is a coordinate system $\yb^{(i)}\eq(y^{(i)}_1,\dots,y^{(i)}_n)$ of the complement $V_i$ of the proper transform of $x_i^{-1}(0)$ in $\Xt$ such that
\begin{equation} \label{3.1}
\aligned&y^{(i)}_i\eq\pi^*x_i,\q y^{(i)}_j\eq\pi^*(x_j/x_i)\,\,\,(j\nes i),\\
&\pi^*\!f\eq(y^{(i)}_i)^d\one h_i(\widehat{\yb}^{(i)}),\q\pi^*\ddd\xb\eq(y^{(i)}_i)^{n-1}\one\ddd\yb^{(i)},\endaligned
\end{equation}
with $h_i(\widehat{\yb}^{(i)})\defs f|_{x_i=1}$. Here $x_j$ for $j\nes i$ is identified with $y^{(i)}_j$ in the definition of $h_i(\widehat{\yb}^{(i)})$, and $\widehat{\yb}^{(i)}\defs(y^{(i)}_1,\dots,y^{(i)}_{i-1},y^{(i)}_{i+1},\dots,y^{(i)}_n)$ for $i\ins[2,n{-}1]$ (similarly for $i\eq1$ or $n$). We then get
\begin{equation} \label{3.2}
{\rm rlct}(f)\eq\tfrac{n}{d},
\end{equation}
by Theorem\,\,\ref{T1}.
\sk
Let $E_i$ be the complement in $E$ of the intersection of $E$ with the proper transform of $x_i^{-1}(0)$ so that the $\widehat{\yb}^{(i)}$ are coordinates of $E_i$. We have the decomposition
\begin{equation} \label{3.3}
V_i=E_i{\times}\one\R,
\end{equation}
in a compatible way with the coordinates $\widehat{\yb}^{(i)}$ and $y^{(i)}_i$.
For $i\ins[1,n]$, set
\begin{equation*}
U'_i\defs\bl\{|y^{(i)}_j|\slt1\pl\ep\,\,\,(\forall\one j\nes i)\br\}\subset E_i,
\end{equation*}
for $\ep\ins\R_{>0}$ sufficiently small. This gives an open covering of $E$. Let $\psi_i$ ($i\ins[1,n]$) be a partition of unity on $E$ which is subordinate to the covering $\{U'_i\}$. Let $\eta$ be a cutoff function on $\R$. Here we assume that $\eta$ is an {\it even function} (that is, $\eta(-y^{(i)}_i)\eq\eta(y^{(i)}_i)$). Set
\begin{equation*}
\chit\defs\msum_{i=1}^n\,\eta(y^{(i)}_i)\one\psi_i(\widehat{\yb}^{(i)}).
\end{equation*}
The restriction of $\chit$ to a sufficiently small neighborhood of $E$ is equal to the constant function with value 1. So there is a cutoff function $\chi'$ on $X$ with $\pi^*\chi'\eq\chit$. Applying Theorem\,\,\ref{T2.1} to $\vp\mi\vp(0)\chi'$ (and using a remark after \eqref{2.1}), the assertion is reduced to
\begin{equation} \label{3.4}
\be(f,\chi')<-\tfrac{n}{d}.
\end{equation}
In the $d$ even case it is well sufficient to show for any $i\ins[1,n]$ the vanishing
\begin{equation} \label{3.5}
\int_{E_i}\!\int_{\R}e^{\sqrt{-1}\one\tau y_i^dh_i(\widehat{\yb}^{(i)})}y_i^{n-1}\one\eta(y_i)\one\psi_i(\widehat{\yb}^{(i)})\one\ddd y_i\one\ddd\widehat{\yb}^{(i)}=0,
\end{equation}
by the definition of $\chi'$ (since $f$ is homogeneous). Here $y_i^{(i)}$ is denoted by $y_i$ to simplify the notation. This vanishing is however trivial using the double integral and calculating the integral over $\R$ first, since $d,n$ are even, and $\eta$ is an even function (and the integration of an odd function with compact support on $\R$ vanishes).
Note that this is well compatible with \cite[Section 7.1, (ii)]{CKN}.
\sk
In the $d$ {\it odd\one} case we can apply the above argument to the {\it real\one} part of $I(\tau,\chi')$, where we have the identity
\begin{equation*}
e^{\sqrt{-1}\one\theta}\eq\cos\theta\pl\sqrt{-1}\one\sin\theta\q(\theta\ins\R),
\end{equation*}
with $\cos\theta$, $\sin\theta$ respectively {\it even\one} and {\it odd\one} functions of $\theta$. We can show the vanishing of the first leading terms with respect to $\al$ of the asymptotic expansion of the {\it imaginary\one} part of $I(\tau,\chi')$ (see \eqref{1}--\eqref{2}) using the involution $\iota$ of $X$ defined by $\iota^*x_i\eq{-}x_i$ for any $i\ins[1,n]$. Indeed, we have $\iota^*f\eq{-}f$ and we may assume $\iota^*\chi'\eq\chi'$ replacing $\chi'$ with $(\chi'\pl\iota^*\chi')/2$ and applying Corollary\,\,\ref{C2.1}. (Here $\iota^*\ddd\xb\eq(-1)^n\ddd\xb$, but we do not have to use the condition that $n$ is even, since an {\it orientation\one} is needed to define the integration of a form of highest degree. Similarly a symmetry of $f$ does not imply the vanishing of an oscillatory integral using the action of a permutation of two coordinates.
Note also that the values of each $h_i$ coincide with $\R$ since $d$ is odd, and there might be cancellations among integrals over various $V_i$, if one tries to apply the above argument to the $\sin\theta$ part in order to show the non-vanishing.
It seems quite difficult show the non-vanishing without assuming non-negativity or non-positivity of the values of $f$, that is, $f^{-1}(0)\eq\{0\}$.)
This finishes the proof of Theorem~\ref{T3} (i).
\end{proof}

\begin{proof}[Proof of Theorem~{\rm\ref{T3} (ii)}]
Assume now $n$ is {\it odd\one} and $f^{-1}(0)\eq\{0\}$ (hence $d$ is {\it even\one}). We may assume further that the values of $f$ are {\it non-negative.} By a more precise version of Theorem~\ref{T2} in \cite{V} we have
\begin{equation} \label{3.6}
\lim_{\tau\one\to\one+\infty}\tau^{n/d}I(\tau,\chi')=C_{-n/d,0}(\chi'),
\end{equation}
where $C_{-n/d,0}(\chi')$ is as in the asymptotic expansion \eqref{2} with $\chi'$ as constructed above. Indeed, by the definition of (weak) asymptotic expansion we have
\begin{equation*}
\lim_{\tau\one\to\one+\infty}\tau^{n/d}\bl(I(\tau,\chi')-C_{-n/d,0}(\chi')\one\tau^{-n/d}\one\br)=0.
\end{equation*}
Here the intersection of the multi-diagonal line (containing the origin and $(1,\dots,1)$) with the Newton boundary $\dd\Gap(f)$ is contained in the interior of the unique maximal-dimensional compact face of $\Gap(f)$.
\sk
It is then enough to show the non-vanishing
\begin{equation} \label{3.7}
\lim_{\tau\one\to\one+\infty}\tau^{n/d}I(\tau,\chi')\ne0.
\end{equation}
Setting
\begin{equation*}
I_1\bl(\tau,y^{n-1}\eta(y)\br):=\int_{\R}e^{\sqrt{-1}\one\tau y^d}y^{n-1}\eta(y)\ddd y,
\end{equation*}
the left-hand side of \eqref{3.5} is equal to
\begin{equation} \label{3.8}
\int_{E_i}I_1\bl(h_i(\widehat{\yb}^{(i)})\tau,y_i^{n-1}\eta(y_i)\br)\psi_i(\widehat{\yb}^{(i)})\one\ddd\widehat{\yb}^{(i)}.
\end{equation}
By \cite[7.1 (i)]{CKN} we have
\begin{equation} \label{3.9}
\lim_{\tau\one\to\one+\infty}\tau^{n/d}I_1\bl(\tau,\chi'y^{n-1}\eta(y)\br)=C'_{-n/d,0}\ne0.
\end{equation}
Set $A\defs C'_{-n/d,0}\ins\C^*$. For any $B\ins(0,1)$, there is a positive real number $R$ such that
\begin{equation} \label{3.10}
\Re\bl(A^{-1}I_1\bl(\tau,y^{n-1}\eta(y)\br)\br)>B\one\tau^{-n/d}\q\q\h{if}\,\,\,\,\tau>R,
\end{equation}
where $\Re(z)$ denotes the real part of a complex number $z$. Since $f(X\stm\{0\})\eq\R_{>0}$ and $h_i\eq f|_{x_i=1}$, there are positive real numbers $M,N$ satisfying
\begin{equation} \label{3.11}
0<M\les h_i(\widehat{\yb}^{(i)})\les N\q\q\h{for any}\,\,\,\,\widehat{\yb}^{(i)}\in\overline{U'_i}.
\end{equation}
For each $i\ins[1,n]$, these imply that if $\tau>R/M$ (so that $h_i(\widehat{\yb}^{(i)})\tau\sgt R$ for any $\widehat{\yb}^{(i)}\ins\overline{U'_i}\,$), we have
\begin{equation} \label{3.12}
\begin{aligned}
&\Re\Bigl(A^{-1}\int_{E_i}I_1\bl(h_i(\widehat{\yb}^{(i)})\tau,y_i^{n-1}\eta(y_i)\br)\psi_i(\widehat{\yb}^{(i)})\one\ddd\widehat{\yb}^{(i)}\Bigr)\\
&>B\int_{E_i}\bl(h_i(\widehat{\yb}^{(i)})\tau\br)^{-n/d}\psi_i(\widehat{\yb}^{(i)})\one\ddd\widehat{\yb}^{(i)}\\&\ges BN^{-n/d}L_i\,\tau^{-n/d}>0,
\end{aligned}
\end{equation}
setting $L_i\defs\int_{E_i}\psi_i(\widehat{\yb}^{(i)})\one\ddd\widehat{\yb}^{(i)}\sgt0$.
So the assertion \eqref{3.7} follows. (One can verify that there is no problem about sign by restricting to the complement of the union of $\overline{U'_j}$ for $j\nes i$, since the values of the $h_i$ are positive.) This finishes the proof of Theorem~\ref{T3} (ii).
\end{proof}

\begin{rem} \label{R3.1}
One could get that $\be(f)\eq{-}\tfrac{n{+}1}{d}$ under the hypotheses of Theorem\,\,\ref{T3} (i) with $f^{-1}(0)\eq\{0\}$ by using \cite[7.1 (i)]{CKN}.
Note that the complement of $f^{-1}(0)$ cannot be connected if $d$ is {\it odd.} (Indeed, the two open subsets $f^{-1}(\R_{>0})$ and $f^{-1}(\R_{<0})$ are both non-empty using the involution $\iota$.)
\end{rem}

\begin{rem} \label{R3.2}
The above argument might be extended to the case $f$ is convenient, Newton-$\R$-nondegenerate, and nonnegative-valued (without assuming that $f$ is homogeneous). Let $\si_j$ ($j\ins J_0$) be the $(n{-}1)$-dimensional faces of $\Gap(f)$ containing the intersection of $\dd\Gap(f)$ with the diagonal spanned by $\bo\ins\R^n$. Let $d_j$ be the LCM of the denominators of the coefficients of the linear function $\ell_j$ on $\R^n$ such that $\si_j\sst\ell_j^{-1}(1)$. Set $r_j\eq d_j\ell_j(\bo)$. (Note that the pullback of a monomial $\xb^{\nu}$ with $d_j\ell_j(\nu)\eq k$ has zero of order $k$ at general point of a divisor corresponding to $\si_j$, and the pullback of a logarithmic form $\xb^{-\bo}\ddd\xb$ is a logarithmic form.)
Assuming $1\sgt r_j/d_j\eq\ell_j(\bo)$ (which is independent of $j\ins J_0$), the equality could hold in \eqref{5} with $\ga(f)\eq r_j/d_j$ if $d_j$ is even and $r_j$ is odd for some $j\ins J_0$. (Without the assumption on the non-negativity of $f$, it does not seem easy to show the {\it non-cancellation\one} among the integrals associated with different faces $\si_j$ or between the integrals over $\{h_{\si}\sgt0\}$ and $\{h_{\si}\slt0\}$.) Here one may have also a problem about the commutativity of asymptotic expansion and integration.
\end{rem}

\begin{rem} \label{R3.3}
There is a big difference between Newton-nondegeneracy over $\R$ and $\C$; for instance in the case $f$ is homogeneous, convenient, and $f^{-1}(0)\eq\{0\}$, the $m${\one}th power $f^m$ is Newton-$\R$-nondegenerate for any positive integer $m$ if $f$ is.
There is also an example like $f\eq(x^2\pl y^2\pl z^2)^2\pl x^6\pl y^6\pl z^6$, which is Newton-nondegenerate over $\R$, but not over $\C$. Here ${\rm rlct}(f)\eq\tfrac{3}{4}$, ${\rm lct}(f_{\C})\eq\tfrac{2}{3}$ (using Remark\,\,\ref{R1.3}), and $\tfrac{3}{4}$ is {\it not\one} a spectral number of $f_{\C}$ according to Singular \cite{Sing} (where the spectral numbers are shifted by $-1$).
\end{rem}

\end{document}